\documentclass{amsproc}
\usepackage{graphicx}
\usepackage{amssymb}
\usepackage{epstopdf}
\usepackage{amsmath,amsthm,amscd,amssymb}
\usepackage{latexsym}
\usepackage[colorlinks,citecolor=red,pagebackref,hypertexnames=false]{hyperref}
\usepackage{geometry}                
\numberwithin{equation}{section}

\theoremstyle{plain}
\newtheorem{theorem}{Theorem}[section]
\newtheorem{lemma}[theorem]{Lemma}

\theoremstyle{definition}
\newtheorem{definition}[theorem]{Definition}

\newtheorem{case[theorem]}{Case}

\theoremstyle{remark}

\numberwithin{equation}{section}

\def\dH{\dim_{{\mathcal H}}}
\def\R{\Bbb R}
\def\a{\alpha}

\begin{document}

\title{Falconer distance problem, additive energy and Cartesian products} 


\author{Alex Iosevich and Bochen Liu}

\date{today}

\email{iosevich@math.rochester.edu}
\email{bochen.liu@rochester.edu}
\address{Department of Mathematics, University of Rochester, Rochester, NY}


\maketitle

\begin{abstract} A celebrated result due to Wolff says if $E$ is a
  compact subset of ${\Bbb R}^2$, then the Lebesgue measure of the
  distance set $\Delta(E)=\{|x-y|: x,y \in E \}$ is positive if the
  Hausdorff dimension of $E$ is greater than $\frac{4}{3}$. In this
  paper we improve the $\frac{4}{3}$ barrier by a small exponent for
  Cartesian products. In higher dimensions, also in the context of
  Cartesian products, we reduce Erdogan's $\frac{d}{2}+\frac{1}{3}$ exponent to $\frac{d^2}{2d-1}$. The proof uses a combination of Fourier analysis and additive comibinatorics. \end{abstract} 


\section{Introduction}

\vskip.125in

The Falconer distance conjecture (\cite{Fal86}) says that if the Hausdorff dimension of $E \subset {\Bbb R}^d$, $d \ge 2$, is greater than $\frac{d}{2}$, then the Lebesgue measure of the distance set $\Delta(E)=\{|x-y|: x,y \in E \}$ is positive. 

The best known results are due to Wolff (\cite{W99}) in two dimensions and Erdogan (\cite{Erd05}) in higher dimensions. They proved that the Lebesgue measure of $\Delta(E)$ is positive if the Hausdorff dimension of $E$ is greater than $\frac{d}{2}+\frac{1}{3}$. This was accomplished by showing that if $s \in \left( \frac{d}{2}, \frac{d+2}{2} \right)$ is the Hausdorff dimension of $E$ and $\mu$ is a Frostman measure on $E$ which has finite $(s-\epsilon)$-energy $I_{s-\epsilon}(\mu)$ for all $\epsilon>0$, then for all $\epsilon>0$,
\begin{equation} \label{erdest} \int_{S^{d-1}} {|\widehat{\mu}(t \omega)|}^2 d\omega \leq C t^{-\frac{d+2s-2}{4}+\epsilon}.\end{equation} 

In particular, in the two-dimensional case, which is the focus of this paper, the estimate takes the form 
\begin{equation} \label{erdest2} \int_{S^1} {|\widehat{\mu}(t \omega)|}^2 d\omega \leq C t^{-\frac{s}{2}+\epsilon}.  \end{equation}

This estimate is then plugged into the Mattila integral, 
\begin{equation} \label{mattilaint} {\mathcal M}(\mu)=\int_1^{\infty} {\left( \int_{S^{d-1}} {|\widehat{\mu}(t \omega)|}^2 d\omega \right)}^2 t^{d-1} dt, \end{equation} the most effective tool developed so far for the study of the Falconer distance problem. 

Mattila proved in \cite{Mat87} that if $E$ is a compact set of Hausdorff dimension $s>\frac{d}{2}$ and $\mu$ is Borel measure supported on $E$ such that ${\mathcal M}(\mu)<\infty$, then the Lebesgue measure of $\Delta(E)$ is positive. For discussion about other versions of Mattila integrals, see \cite{GILP13}.
\vskip.125in
Our results are the following. 

\begin{theorem} \label{main1} Let $E=A \times B$, where $A$ and $B$
  are compact subsets of ${\Bbb R}$ with positive $s_A, s_B$-dimensional Hausdorff measure, respectively. If $s_A+s_B+\max(s_A,s_B)>2$, the Lebesgue measure of
  $\Delta(E)$ is positive. In particular, if $\dH(E)=\dH(A)+\dH(B)$ and $\dH(A)\neq \dH(B)$,
  $\dH(E)>\frac{4}{3}-\frac{|\dH(A)-\dH(B)|}{3}$ implies $\Delta(E)$ has positive Lebesgue measure.
\end{theorem} 

\vskip.125in 

To state the result in the case $\dH(A)=\dH(B)$, we need the following definition.

\begin{definition}\label{regC}
  Let $A$ be a compact subset of $\R^d$ of Hausdorff dimension
  $s_A$. We say $A$ is Ahlfors-David regular if there exists a Radon
  measure $\nu_A$ on $A$ and a constant $0<C_{\nu_A}<\infty$ such that
\begin{equation}\label{AD}
C_{\nu_A}^{-1}r^{s_A}<\nu_A(B(x,r))<C_{\nu_A}r^{s_A},\ \forall x\in
A,\,0<r<1.
\end{equation}

\end{definition}
\vskip.125in 
\begin{theorem}\label{main2}
Suppose $E=A \times B$, $s_A=s_B=\alpha$ and $A$ is
Ahlfors-David regular with $\nu_A$, $C_{\nu_A}$ such that \eqref{AD} holds. Then there exists $\delta=\delta(C_{\nu_A})>0$ such that
whenever $\a>\frac{2}{3}-\delta$, the Lebesgue measure of
  $\Delta(E)$ is positive.
\end{theorem}
We also obtain an improvement of Erdogan's $\frac{d}{2}+\frac{1}{3}$ exponent in higher dimension for Cartesian products. 
\begin{theorem} \label{mind} Suppose that $E$ is a compact subset of ${\Bbb R}^d$ of the form $A_1 \times A_2 \times \dots \times A_d$, where $A_j \subset {\Bbb R}$ has positive $s_j$-dimensional Hausdorff measure for all $1\leq j\leq d$. Suppose that $\sum_{j=1}^d s_j>\frac{d^2}{2d-1}$. Then the Lebesgue measure of $\Delta(E)$ is positive. \end{theorem} 

\vskip.125in 

\subsection{Acknowledgements} The authors wish to thank Josh Zahl and the anonymous referee for several useful suggestions. 

\vskip.125in 

\subsection{Outline of the proof of Theorems \ref{main1}, \ref{main2} and \ref{mind}} Our argument consists of three basic steps. 

\vskip.125in 

\begin{itemize} 
\item We first establish Theorem \ref{main1} which is accomplished using the imbalance inherent in the structure of the Mattila Integral. 

\vskip.125in 

\item The improvement of the $\frac{d}{2}+\frac{1}{3}$ exponent for Cartesian products in higher dimensions (Theorem \ref{mind}) is accomplished in the same way regardless of whether the Hausdorff dimension of the fibers is the same. 

\vskip.125in 

\item In the case when $dim_{{\mathcal H}}(A) =dim_{{\mathcal H}}(B)$,
  we use a recent result due to Semyon Dyatlov and Josh Zahl
  (\cite{Z15}) to show that when $A$ is Ahlfors-David regular, the additive energy of $A$ at scale $t^{-1}$, 
$$ \nu_A^4 \{(a_1,a_2,a_3,a_4): |(a_1+a_2)-(a_3+a_4)| \leq t^{-1} \}, $$ where $\nu_A$ is a Frostman measure on 
$A$, satisfies a better than trivial estimate, namely $Ct^{-dim_{{\mathcal H}}(A)-\delta}$ for some $\delta>0$, and then show that this leads to a slightly better exponent than  $\frac{4}{3}$. 

\vskip.125in

\end{itemize} 

\vskip.25in 

\section{Proof of Theorem \ref{main1}} 

\vskip.125in 

We shall repeatedly use the following simple estimate. 

\begin{lemma} \label{basic} (Solid Average) Suppose that $\nu$ is a compactly supported Borel measure on ${\Bbb R}^d$ such that $\nu(B(x,r)) \leq Cr^{\alpha}$ for all $x \in {\Bbb R}^d$. Then for any bounded rectangle $R$, 
$$ \int_R {|\widehat{\nu}(tu)|}^2 du \leq C_R t^{-\alpha}.$$ 
\end{lemma} 

To prove the lemma observe that the left hand side is 
$$ \leq \int {|\widehat{\nu}(tu)|}^2 \widehat{\psi}(u) du, $$ where $\psi$ is a suitably chosen smooth compactly supported function. This expression equals 
$$ \int \int \int e^{2 \pi i(x-y) \cdot tu} \widehat{\psi}(u) du d\nu(x) d\nu(y)=\int \int \psi(t(x-y)) d\nu(x) d\nu(y) \leq C_R t^{-\alpha}$$ by assumption. This completes the proof of the lemma. 

\vskip.125in 

We now parameterize the upper semi-circle $S_1^+$ in the form 
$$\left\{(u, \sqrt{1-u^2}): -1 \leq u \leq 1 \right\}.$$ The argument shall be carried out for this parameterization as the proof for the lower semi-circle is identical. 

Let $d\mu(x)=d\nu_A(x_1) d\nu_B(x_2)$, where $\nu_A, \nu_B$ are Frostman probability measures on $A$ and $B$, respectively such that 
$$\nu_A(B(x,r)) \leq C r^{s_A}, \ \ \nu_B(B(x,r)) \leq C r^{s_B}.$$
Assume without loss of generality that $s_A \ge s_B$. Also assume $E$ is not a point mass, which implies that either
$$\exists \,a\in\R, \ \mu(\{(x_1,x_2):x_1>a\}),\mu(\{(x_1,x_2):x_1<a\})>0,$$
or
$$\exists \,b\in \R,\ \mu(\{(x_1,x_2):x_2>b\}),\mu(\{(x_1,x_2):x_2<b\})>0.$$ 

Without loss of generality, we may assume $\mu(\{(x_1,x_2):x_2>b\}),\mu(\{(x_1,x_2):x_2<b\})>0$ for some $b\in\R$. It follows that
\begin{equation} \label{horsep} \iint \frac{|x_2-y_2|}{|x-y|} d\mu(x) d\mu(y) >0. \end{equation} 
Let $\omega=(\cos(\theta), \sin(\theta))$. Consider the modified Mattila Integral 
\begin{equation} \label{matmod} \int {\left( \int_{S^1} {|\widehat{\mu}(t\omega)|}^2 |\sin(\theta)| d\omega \right)}^2 t dt. \end{equation}

\begin{lemma}\label{modMat}
Suppose \eqref{horsep} holds. Then the finiteness of the integral \eqref{matmod} implies that the Lebesgue measure of the distance set is positive.
\end{lemma}

\begin{proof}[proof]
To prove this lemma, one simply replaces the distance measure in the derivation of the Mattila Integral in \cite{W99}, given by the relation 
$$ \int f(t) d\nu^*_0(t)=\int \int f(|x-y|) d\mu(x) d\mu(y)$$ by a slightly modified distance measure given by 
$$ \int f(t) d\nu_0(t)=\int \int f(|x-y|)  \frac{|x_2-y_2|}{|x-y|} d\mu(x) d\mu(y).$$ 
\vskip.125in 
As in \cite{W99}, define
$$d\nu(t)=e^{i\frac{\pi}{4}} t^{-\frac{1}{2}} d\nu_0(t)+e^{-i\frac{\pi}{4}} |t|^{-\frac{1}{2}} d\nu_0(-t)$$
and it follows that
\begin{equation}\label{hatnu}
\widehat{\nu}(t)=\iint |x-y|^{-\frac{1}{2}}\cos(2\pi (t|x-y|-\frac{1}{8})) \frac{|x_2-y_2|}{|x-y|} d\mu(x) d\mu(y).
\end{equation}

On the other hand,

\begin{equation}\label{hatmu1}
\int |\widehat{\mu}(t\omega)|^2|\sin\theta|\,d\theta=\iint\left(\int e^{2\pi i (x-y)\cdot(t\omega)}|\sin\theta|\,d\theta\right)d\mu(x)\,d\mu(y). 
\end{equation}

Let $\theta_{x-y}$ be the angle between the vector $x-y$ and the $x$-axis. Then $|\sin\theta_{x-y}|=\frac{|x_2-y_2|}{|x-y|}$.  We may assume $s$, the Hausdorff dimension of $E$, is not greater than $\frac{3}{2}$. By stationary phase(see, e.g. \cite{W04} for details),

\begin{equation}\label{hatmu2}
\begin{aligned}
 \int e^{2\pi i (x-y)\cdot(t\omega)}|\sin\theta|\,d\theta= &2 (|t||x-y|)^{-\frac{1}{2}} \cos(2\pi (t|x-y|-\frac{1}{8}))|\sin \theta_{x-y}|+ O((t|x-y|)^{-\frac{3}{2}})\\=&2 (|t||x-y|)^{-\frac{1}{2}} \cos(2\pi (t|x-y|-\frac{1}{8}))\frac{|x_2-y_2|}{|x-y|}+O((t|x-y|)^{-s+\epsilon}).
\end{aligned}
\end{equation}

\vskip.125in
Putting \eqref{hatnu}, \eqref{hatmu1}, \eqref{hatmu2} together, one can see

\begin{equation*}
\begin{aligned}
||\widehat{\nu}||_2^2= & \int_{|t|\leq 1} |\widehat{\nu}(t)|^2\,dt+\int_{|t|\geq 1} |\widehat{\nu}(t)|^2\,dt\\\leq & 1+\int_1^\infty \left(\int |\widehat{\mu}(t\omega)|^2|\sin\theta|\,d\theta\right)^2 t\,dt+C I_{s-\epsilon}(\mu),\end{aligned}
\end{equation*}

which proves the lemma since $I_{s-\epsilon}(\mu)<\infty$.

\end{proof}

We now proceed with the estimation of (\ref{matmod}). It follows that 
$$ \int_{S^1} {|\widehat{\mu}(t\omega)|}^2 |\sin(\theta)| d\omega$$ is bounded by the sum of two terms of the form 
\begin{equation} \label{avAB} 
\begin{aligned}
& \int_{-1}^1 {|\widehat{\nu}_A(tu)|}^2 {\left|\widehat{\nu}_B\left(\pm t \sqrt{1-u^2}\right)\right|}^2 du
\\\leq &\int_{-1}^1 {|\widehat{\nu}_A(tu)|}^2 du \leq C t^{-s_A} \end{aligned}\end{equation} by Lemma \ref{basic}. 

Plugging (\ref{avAB}) into the modified Mattila integral (\ref{matmod}) we see that 
\begin{equation*}\begin{aligned}{\mathcal M}(\mu) \leq &C \int \int_{S^1} {|\widehat{\mu}(t\omega)|}^2 t \cdot t^{-s_A} d\omega dt\\
=&C \int {|\widehat{\mu}(\xi)|}^2 {|\xi|}^{-s_A} d\xi\\
=& C'\int \int {|x-y|}^{-2+s_A} d\mu(x) d\mu(y)\end{aligned}\end{equation*} and this energy integral (see e.g. \cite{W04} or \cite{M95}) is finite if 
$$ s_A+s_B>2-s_A,$$
 as desired. 

\vskip.25in 

\section{Proof of Theorem \ref{main2}} 

We improve the upper bound of
\eqref{erdest2} to prove the theorem. More precisely, under the
assumptions of Theorem \ref{main2}, there exists $\delta=\delta(C_{\nu_A})>0$
such that

$$\int_{S^1}|\hat{\mu}(t\omega)|^2\,d\omega\lesssim t^{-\alpha-\delta},$$
where $\mu=\nu_A\times\nu_B$.
\vskip.125in
First we deal with the case when $\theta$ is close to $0$. We have 
\begin{equation} \label{smallangle}\begin{aligned}& \int_0^{\delta} {|\widehat{\nu}_A(t\cos(\theta))|}^2 {|\widehat{\nu}_B(t \sin(\theta))|}^2 d\theta 
\\ \leq& \int {|\widehat{\nu}_B(tu)|}^2 \widehat{\psi}(u/\delta) du\end{aligned}\end{equation} with an appropriately chosen cut-off function $\psi$. This expression equals 
$$ \delta \int \int \psi(\delta t(u-v)) d\nu_B(u) d\nu_B(v) \leq C  \delta^{1-\alpha} \cdot t^{-\alpha}.$$ 

Choosing $\delta=t^{-\gamma_0}$, where $\gamma_0$ is a small positive number to be determined later, we see that the expression in (\ref{smallangle}) is 
\begin{equation}
  \label{near0}
 \leq C t^{-\gamma_0(1-\alpha)} \cdot t^{-\alpha}. 
\end{equation}
\vskip.125in
We can deal with the neighborhood near $\frac{\pi}{2}$ in the same way, so we omit this part of the calculation. 
\vskip.125in
Now consider 
\begin{equation} \label{start} \int_I {|\widehat{\nu}_A(t\cos(\theta))|}^2 {|\widehat{\nu}_B(t\sin(\theta))|}^2 d\theta, \end{equation} where 
$I$ is an interval that excludes both $(0, t^{-\gamma_0})$ and a fixed neighborhood of $\frac{\pi}{2}$. 

By Cauchy-Schwartz, this expression (\ref{start}) is bounded by 
$$ C {\left( \int_I {|\widehat{\nu}_A(t\cos(\theta))|}^4 d\theta \right)}^{\frac{1}{2}} \cdot  {\left( \int_I {|\widehat{\nu}_B(t\sin(\theta))|}^4 d\theta \right)}^{\frac{1}{2}}.$$ 

Making the change of variables $u=\cos(\theta)$ and $u=\sin(\theta)$, respectively, we see that this expression is 
$$ \leq C t^{\gamma_0} {\left( \int {|\widehat{\nu}_A(tu)|}^4 \widehat{\psi}(u) du \right)}^{\frac{1}{2}} \cdot  
{\left( \int  {|\widehat{\nu}_B(tu)|}^4 \widehat{\psi}(u) du
  \right)}^{\frac{1}{2}}=C t^{\gamma_0} \sqrt{I} \cdot
\sqrt{II},$$ where $\psi$ is a smooth positive function whose Fourier transform has compact support. 

Expanding each expression and changing the order of integration, we obtain 
\begin{equation}
\begin{aligned}
I&=\int \int \int \int \psi(t(u_1-u_2+u_3-u_4))
d\nu_A(u_1)d\nu_A(u_2)d\nu_A(u_3)d\nu_A(u_4)
\\&\lesssim
\nu_A\times\nu_A\times\nu_A\times\nu_A\{(u_1,u_2,u_3,u_4)\in A^4: |(u_1+u_2)-(u_3+u_4)|<t^{-1}\}
\end{aligned}
\end{equation}
 and 
\begin{equation}
\begin{aligned}
 II&=\int \int \int \int \psi(t(u_1-u_2+u_3-u_4))
 d\nu_B(u_1)d\nu_B(u_2)d\nu_B(u_3)d\nu_B(u_4)
\\&\lesssim
\nu_B\times\nu_B\times\nu_B\times\nu_B\{(u_1,u_2,u_3,u_4)\in B^4: |(u_1+u_2)-(u_3+u_4)|<t^{-1}\}.
\end{aligned}
\end{equation}

Observe that we trivially have 

\begin{equation}\label{triE}
 I \lesssim t^{-\alpha}; \ II \lesssim t^{-\alpha}.
 \end{equation}

It follows that 
$$ Ct^{\frac{\gamma_0}{2}} \sqrt{I} \cdot \sqrt{II} \leq
Ct^{-\alpha}\leq Ct^{\frac{\gamma_0}{2}} \cdot t^{-\frac{dim_{{\mathcal H}}(A \times B)}{2}},$$ which recovers Wolff's $\frac{4}{3}$ exponent as $\gamma_0 \to 0$. Moreover, the only way this estimate does not beat 
$\frac{4}{3}$ is if 
\begin{equation} \label{lenergy} I,II \ge
  Ct^{-\alpha+\frac{\gamma_0}{2}} \end{equation} for a sequence of
$t$'s going to infinity. The following theorem due to Dyatlov and Zahl
(\cite{Z15}) shows that this cannot happen for Ahfors-David regular sets. 

\begin{definition}[Dyatlov and Zahl, \cite{Z15}]
  Let $X\subset [0,1]^d$ and $\nu$ be an outer measure
  on $X$ with $0<\nu(X)<\infty$. For $r>0$, define the (scale $r$)
  additive energy by
$$\mathcal{E}(X,\nu,r)=\nu\times \nu\times \nu\times \nu\{(u_1,u_2,u_3,u_4)\in X^4:
|(u_1+u_2)-(u_3+u_4)|<r\}.$$
\end{definition}
\begin{theorem}[Dyatlov and Zahl, \cite{Z15}]\label{Z15}
  Let $X\subset [0,1]$ be an Ahlfors-David regular set of Hausdorff
  dimension $\alpha$ and $\nu$ be a measure
  on $X$ such that for some constant $0<C_\nu<\infty$,
$$C_\nu^{-1}r^\alpha<\nu(B(x,r))<C_\nu r^\alpha,\ \forall x\in X,\,0<r<1.$$
Then
$$\mathcal{E}(X,\nu,r)\leq \widetilde{C} \,r^{\alpha+\beta_\nu}$$
for some $\beta_\nu>0$ and some $\widetilde{C}>0$. In particular, we can
choose
$$\beta_\nu=\alpha\,e^{-exp[K(1+\log C_\nu)^{1/2}(1-\alpha)^{-1/2}]}$$
where $K$ is an absolute constant;  $\widetilde{C}$ depends only on
$\alpha$ and $C_\nu$.
\end{theorem}

From Theorem \ref{Z15}, Definition \ref{regC} and the trivial estimate
\eqref{triE} of $II$, it follows that

$$I\lesssim t^{-\alpha-\beta_{\nu_A}}, II\lesssim t^{-\alpha},$$
where $\beta_{\nu_A}$ is defined in Theorem \ref{Z15}. All implicit
constants are finite, independent on $t$.
\vskip.125in
Together with the estimate near $\theta=0$ \eqref{near0}, we can bound 
\eqref{erdest2} by
$$Ct^{-\gamma_0(1-\alpha)-\alpha}+Ct^{\frac{\gamma_0}{2}}t^{-\alpha-\gamma},$$
where $\gamma=\gamma(C_{\nu_A})>0$. Let $\gamma_0>0$ be a small enough, we get
$$\int_{S^1}|\hat{\mu}(t\xi)|^2\,d\xi\lesssim t^{-\alpha-\delta}$$
for some $\delta=\delta(C_{\nu_A})>0$.

\section{Proof of Theorem \ref{mind}} 

\vskip.125in 

Let $\nu_j$ denote the restriction of the $s_j$-dimensional Hausdorff measure to $A_j$ and assume without loss of generality that $s_1 \ge s_2 \ge \dots \ge s_d$. Parameterize the upper half-sphere in the form 
$$ \left\{(u_1, u_2, \dots, u_{d-1}, \sqrt{1-u_1^2-\dots-u_{d-1}^2}): -1 \leq u_j \leq 1 \right\}.$$ 

Let $\mu$ denote the product measure on $E$, $\theta_\omega$ be the angle between the vector $\omega\in S^{d-1}$ and the hyperplane $\{x_{d}=0\}$. Without loss of generality, we may assume 
$$ \mu(\{(x_1,\dots, x_d):x_d>a\}),\mu(\{(x_1,\dots,x_d):x_d<a\})>0$$ for some $a\in\R$.

An argument identical to the one in the proof of Lemma \ref{modMat} shows the finiteness of 
\begin{equation}\label{modMatRd}
 \int_1^\infty\left(\int_{S^{d-1}} {|\widehat{\mu}(t \omega)|}^2 |\sin\theta_\omega|\,d\omega\right)^2 t^{d-1}\,dt 
 \end{equation}
implies that the distance set has positive Lebesgue measure.
\vskip.125in
It follows that $$ \int_{S^{d-1}} {|\widehat{\mu}(t \omega)|}^2 |\sin\theta_\omega|\,d\omega$$ is bounded by two terms of the form
\begin{equation*}
\begin{aligned}
&\int_{-1}^1 \dots \int_{-1}^1 {|\widehat{\nu}_1(tu_1)|}^2 \dots 
{|\widehat{\nu}_{d-1}(tu_{d-1})|}^2 \cdot {\left|\widehat{\nu}_d\left(\pm t\sqrt{1-|u|^2}\right)\right|}^2 du_1 \dots du_{d-1}\\
 \leq& \int_{-1}^1 \dots \int_{-1}^1 {|\widehat{\nu}_1(tu_1)|}^2 \dots {|\widehat{\nu}_{d-1}(tu_{d-1})|}^2 \,du_1 \dots du_{d-1}.
 \end{aligned}
 \end{equation*}

By Lemma \ref{basic}, this quantity is 
$$ \leq Ct^{-s_1-\dots-s_{d-1}} \leq Ct^{-s \frac{d-1}{d}},$$
where $s=\sum_{j=1}^d s_j$.

Plugging this estimate into the Mattila Integral (\ref{modMatRd}) we obtain 
\begin{equation*}\begin{aligned}& C \int \int {|\widehat{\mu}(t \omega)|}^2 t^{d-1} \cdot t^{-s \frac{d-1}{d}} d\omega dt\\
=&C \int {|\widehat{\mu}(\xi)|}^2 {|\xi|}^{-s \frac{d-1}{d}} d\xi\end{aligned}\end{equation*} and this integral is finite if 
$$ d-s<s \frac{d-1}{d},$$ which is the case if 
$$ s>\frac{d^2}{2d-1}, \ \text{as desired}.$$

\vskip.125in 

\end{document}